\documentclass[a4paper, 12pt]{amsart}
\usepackage{amsmath, amssymb, amscd, enumerate, hyperref, MnSymbol, wasysym, tikz, stmaryrd, tensor}
\usepackage[all,cmtip]{xy}

\newtheorem{theorem}{Theorem}[section]
\newtheorem{proposition}[theorem]{Proposition}
\newtheorem{lemma}[theorem]{Lemma}
\newtheorem{corollary}[theorem]{Corollary}

\theoremstyle{remark}
\newtheorem*{remark}{Remark}
\newtheorem{example}{Example}
\newtheorem*{definition}{Definition}

\makeatletter
\def\@eqnnum{(\theequation)}
\makeatother

\numberwithin{equation}{section}

\begin{document}

\title[]
{On the $\Gamma$-equivalence of binary quadratic forms}

\author{Bumkyu Cho}
\address{Department of Mathematics, Dongguk University-Seoul, 30 Pildong-ro 1-gil, Jung-gu, Seoul, 04620, Republic of Korea}

\email{bam@dongguk.edu}


\thanks{The author was supported by NRF-2015R1C1A1A02037526 and the Dongguk University Research Fund of 2017.}


\dedicatory{}

\begin{abstract}
For a congruence subgroup $\Gamma$, we define the notion of $\Gamma$-equivalence on binary quadratic forms which is the same as proper equivalence if $\Gamma = \mathrm{SL}_2(\mathbb Z)$. We develop a theory on $\Gamma$-equivalence such as the finiteness of $\Gamma$-reduced forms, the isomorphism between $\Gamma_0(N)$-form class group and the ideal class group, $N$-representation of integers, and $N$-genus of binary quadratic forms. As an application, we deal with representations of integers by binary quadratic forms under certain congruence condition on variables.
\end{abstract}

\maketitle

\section{Introduction}

\begin{definition}
Let $Q(x, y)$, $Q'(x, y)$ be two binary quadratic forms with coefficients in $\mathbb Z$ and let $\Gamma$ be a congruence subgroup of $\mathrm{SL}_2(\mathbb Z)$. If $Q'(x, y) = Q(ax + by, cx + dy)$ for some $(\begin{smallmatrix} a & b \\ c & d \end{smallmatrix}) \in \Gamma$, then we say that $Q'(x, y)$ is $\Gamma$-equivalent to $Q(x, y)$.
\end{definition}

The $\mathrm{SL}_2(\mathbb Z)$-equivalence is clearly the same as the usual proper equivalence. The purpose of this article is to develop a theory of $\Gamma$-equivalence just as the theory of proper equivalence. To do this, we mainly follow the course given in Cox's book \cite{Cox}.

\medskip

The beginning of this work is based on the following observation. Suppose that $Q(x, y)$ and $Q'(x, y)$ are $\Gamma$-equivalent.

i) If $\Gamma = \Gamma_0(N)$, then
\[
\{ Q(x, y) \in \mathbb Z \, | \, x, y \in \mathbb Z, \, y \equiv 0 \, (\bmod \, N) \} \, = \, \{ Q'(x, y) \in \mathbb Z \, | \, x, y \in \mathbb Z, \, y \equiv 0 \, (\bmod \, N) \}.
\]

ii) If $\Gamma = \Gamma_1(N)$ and $r \in \mathbb Z$, then
\[
\{ Q(x, y) \in \mathbb Z \, | \, x \equiv r \, (\bmod \, N), \, y \equiv 0 \, (\bmod \, N) \} \, = \, \{ Q'(x, y) \in \mathbb Z \, | \, x \equiv r \, (\bmod \, N), \, y \equiv 0 \, (\bmod \, N) \}.
\]

iii) If $\Gamma = \Gamma(N)$ and $r, s \in \mathbb Z$, then
\[
\{ Q(x, y) \in \mathbb Z \, | \, x \equiv r \, (\bmod \, N), \, y \equiv s \, (\bmod \, N) \} \, = \, \{ Q'(x, y) \in \mathbb Z \, | \, x \equiv r \, (\bmod \, N), \, y \equiv s \, (\bmod \, N) \}.
\]

In Section 3, we deal with representations of integers by binary quadratic forms under the congruence condition given in Case i) above.

\section{The $\Gamma$-equivalence of binary quadratic forms}

We denote by $\mathrm{QF}$ the set of all primitive positive-definite binary quadratic forms, and denote by $\mathrm{QT}$ the set of all points $\tau \in \mathbb H$ such that $[\mathbb Q(\tau) : \mathbb Q] = 2$. Then a congruence subgroup $\Gamma$ of $SL_2(\mathbb Z)$ acts on both $\mathrm{QF}$ and $\mathrm{QT}$ in a natural way as
\begin{eqnarray*}
Q \cdot \gamma & = & (x \ y) \tensor*[^t]\gamma{} A_Q \gamma \left(\begin{smallmatrix} x \\ y \end{smallmatrix}\right), \\
\gamma(\tau) & = & \frac{a\tau + b}{c \tau + d},
\end{eqnarray*}
where $Q \in \mathrm{QF}$, $\gamma = \left(\begin{smallmatrix} a & b \\ c & d \end{smallmatrix}\right) \in \Gamma$, $\tau \in \mathrm{QT}$, and $A_Q$ is a symmetric matrix associated with $Q$. By definition, a form $Q'$ is $\Gamma$-equivalent to a form $Q$ if and only if $Q' = Q\cdot \gamma$ for some $\gamma \in \Gamma$.

Given a form $Q \in \mathrm{QF}$, let $\tau_Q$ denote the unique solution of $Q(x, 1) = 0$ in $\mathbb H$. Then it is not difficult to see that the map defined by $Q \mapsto \tau_Q$ is a one-to-one correspondence between $\mathrm{QF}$ and $\mathrm{QT}$, and moreover
\[ \tau_{Q \cdot \gamma} \, = \, \gamma^{-1}(\tau_Q). \]

\begin{definition}
Let $\mathfrak F_\Gamma$ be a fundamental region for a congruence subgroup $\Gamma$.

(1) A form $Q \in \mathrm{QF}$ is said to be $\Gamma$-reduced if $\tau_Q \in \mathfrak F_\Gamma$. The set of all $\Gamma$-reduced forms in $\mathrm{QF}$ is denoted by $\Gamma\text{-}\mathrm{RF}$.

(2) Given a negative integer $D \equiv 0, 1 \pmod 4$, we denote by $\mathrm{QF}(D)$ the set of all forms in $\mathrm{QF}$ of discriminant $D$, and denote by $\Gamma\text{-}\mathrm{RF}(D)$ the set of all $\Gamma$-reduced forms in $\mathrm{QF}(D)$.
\end{definition}

It is worthwhile to remark that the definition of $\Gamma$-reduced forms depends on the choice of the fundamental region $\mathfrak F_\Gamma$.

\begin{example}
As is well known, a fundamental region for $\mathrm{SL}_2(\mathbb Z)$ is
\begin{eqnarray*}
\mathfrak F_{\mathrm{SL}_2(\mathbb Z)} & = & \big\{ \tau \in \mathbb H \, \big| \, |\mathrm{Re}(\tau)| \leq 1/2, \ |\tau| \geq 1, \\
& & \qquad \quad \ \ \big( |\mathrm{Re}(\tau)| = 1/2 \mbox{ or } |\tau| = 1 \big) \Longrightarrow \mathrm{Re}(\tau) \leq 0 \big\}.
\end{eqnarray*}
With respect to this choice of $\mathfrak F_{\mathrm{SL}_2(\mathbb Z)}$, a form $ax^2 + bxy + cy^2 \in \mathrm{QF}$ is $\mathrm{SL}_2(\mathbb Z)$-reduced if and only if it satisfies

(1) $|b| \leq a \leq c$

(2) $(|b| = a \mbox{ or } a = c) \Longrightarrow b \geq 0$.\\
Notice that it coincides with the usual definition of reduced forms for the proper equivalence. We comment that this relation between reduced forms and fundamental domains is not new to mathematicians (see \cite[Section 5.3.1]{Coh} and \cite[Exercise 11.4]{Cox}).
\end{example}

\begin{example}
For $p = 2, 3$ it is not difficult to see that
\begin{eqnarray*}
\mathfrak F_{\Gamma_0(p)} & = & \big\{ \tau \in \mathbb H \, \big| \, |\mathrm{Re}(\tau)| \leq 1/2, \ |\tau - 1/p| \geq 1/p, \ |\tau + 1/p| \geq 1/p, \\
& & \qquad \quad \ \ \big( |\mathrm{Re}(\tau) = 1/2| \mbox{ or } |\tau - 1/p| = 1/p \mbox{ or } |\tau + 1/p| = 1/p \big) \Longrightarrow \mathrm{Re}(\tau) < 0 \big\}.
\end{eqnarray*}
Therefore $ax^2 + bxy + cy^2$ is defined to be $\Gamma_0(p)$-reduced if

(1) $|b| \leq a$, $|b| \leq pc$

(2) $(|b| = a \mbox{ or } |b| = pc) \Longrightarrow b > 0$.\\
Under this defintion, we easily compute that
\begin{eqnarray*}
\Gamma_0(2)\text{-}\mathrm{RF}(-3) & = & \{ x^2 + xy + y^2 \}, \\
\Gamma_0(2)\text{-}\mathrm{RF}(-4) & = & \{ x^2 + y^2, \ 2x^2 + 2xy + y^2 \}, \\
\Gamma_0(2)\text{-}\mathrm{RF}(-7) & = & \{ x^2 + xy + 2y^2, \ 2x^2 \pm xy + y^2 \}, \\
\Gamma_0(2)\text{-}\mathrm{RF}(-8) & = & \{ x^2 + 2y^2, \ 2x^2 + y^2, \ 3x^2 + 2xy + y^2 \},
\end{eqnarray*}
and so on.
\end{example}

\begin{remark}
The reader will see soon that the coefficient of $x^2$ must be relatively prime to $N$ in order for $\Gamma_0(N)\text{-}\mathrm{RF}(D)$ to be an abelian group under the Dirichlet composition.
\end{remark}

\begin{example}
Let $p \geq 5$ be a prime number. Then $ax^2 + bxy + cy^2 \in \mathrm{QF}$ is defined to be $\Gamma_0(p)$-reduced if it satisfies the following properties:

\begin{enumerate}
\item $|b| \leq a$

\item $|b| \leq \frac{p^2c + (k^2 - 1)a}{p \cdot |k|}$ for all $k \in S_p$

\item $|b| = a \ \Longrightarrow \ b = a$

\item $b \neq -pc$

\item $b = - \frac{p^2c + (k^2 - 1)a}{pk}$ for some $k \in E_p^{(2)}$ $\Longrightarrow$ $b \geq -\frac{2ka}{p}$

\item $b = - \frac{p^2c + (k^2 - 1)a}{pk}$ for some $k \in S_p - (\{ \pm 1 \} \cup E_p^{(2)})$ $\Longrightarrow$ $b \geq -\frac{(2k_{(2)} + 1)a}{p}$

\item $b^2 - 4ac = - \frac{3a^2}{p^2}$ $\Longrightarrow$ $b \neq \frac{(1 - 2k)a}{p}$ for all $k \in S_p - (\{ 1 \} \cup E_p^{(3)})$ with $k \neq k_{(3)}$,
\end{enumerate}
where the definitions of $S_p$, $E_p^{(2)}$, $k_{(2)}$, $E_p^{(3)}$, and $k_{(3)}$ are given in Section 4. The definition of $\Gamma_0(p)$-reduction above corresponds to the fundamental region for $\Gamma_0(p)$ described explicitly in Section 4.

From this definition, we can directly deduce that there are at most finitely many $\Gamma_0(p)$-reduced forms of a given discriminant $D$. In fact, it will turn out that the conditions (1) and (2) are good enough to imply the finiteness of $\Gamma_0(p)\text{-}\mathrm{RF}(D)$. Note that
\begin{eqnarray*}
& & |b| \leq \frac{p^2c + (k^2 - 1)a}{p \cdot |k|} \ \mbox{ for all } k \in S_p \\
& \Longrightarrow & \big|\tau_Q - k/p\big| \geq 1/p \ \mbox{ for all } k \in S_p, \mbox{ where } \tau_Q = (-b + \sqrt{D})/2a \\
& \Longrightarrow & \mathrm{Im}(\tau_Q) \geq \sqrt{3}/2 \mbox{ or } \big(|\mathrm{Re}(\tau)| \leq 1/2p, \ |\tau_Q - 1/p| \geq 1/p, \ |\tau_Q + 1/p| \geq 1/p \big) \\
& \Longrightarrow & a \leq \sqrt{-D/3} \mbox{ or } \big(|b| \leq a/p, \ |b| \leq pc \big).
\end{eqnarray*}

Fix a negative discriminant $D$.

\medskip

\underline{Case 1} $|b| \leq a$, $a \leq \sqrt{-D/3}$

Clearly, there are at most finitely many triplets $(a, b, c)$ of integers satisfying these conditions.

\medskip

\underline{Case 2} $|b| \leq a/p$, $|b| \leq pc$

Since $b^2 = |b| \cdot |b| \leq \frac{a}{p} \cdot pc = ac$, the finiteness follows immediately from the fact that
\begin{eqnarray*}
3b^2 \leq 4ac - b^2 = -D & \Longrightarrow & |b| \leq \sqrt{-D/3}, \\
4ac = b^2 - D \leq -4D/3 & \Longrightarrow & ac \leq -D/3.
\end{eqnarray*}

\end{example}

We now prove the finiteness of $\Gamma$-reduced forms of a given discriminant for all $\Gamma$.

\begin{theorem}
Let $\Gamma\text{-}\mathrm{RF}(D)$ be the set of all $\Gamma$-reduced forms of discriminant $D$ with respect to $\mathfrak F_\Gamma$. Then we have the following.

(1) Any two distinct $\Gamma$-reduced forms in $\Gamma\text{-}\mathrm{RF}(D)$ are not $\Gamma$-equivalent.

(2) Every form $Q \in \mathrm{QF}(D)$ is $\Gamma$-equivalent to a unique $\Gamma$-reduced form in $\Gamma\text{-}\mathrm{RF}(D)$.

(3) The cardinality of $\Gamma\text{-}\mathrm{RF}(D)$ is independent of the choice of $\mathfrak F_\Gamma$.

(4) The cardinality $|\Gamma\text{-}\mathrm{RF}(D)|$ is finite.
\end{theorem}

\begin{proof}
(1) Suppose that $Q, Q' \in \Gamma\text{-}\mathrm{RF}(D)$ are $\Gamma$-equivalent to each other. Since $Q' = Q \cdot \gamma$ for some $\gamma \in \Gamma$, we have $\tau_{Q'} = \gamma^{-1}(\tau_Q)$. This says that $\tau_Q, \tau_{Q'} \in \mathfrak F_\Gamma$ are in the same $\Gamma$-orbit, from which we deduce that $\tau_Q = \tau_{Q'}$. Appealing to the one-to-one correspondence between $\mathrm{QF}$ and $\mathrm{QT}$, it follows that $Q =Q'$.

(2) Given a form $Q \in \mathrm{QF}(D)$, there exists $\gamma \in \Gamma$ such that $\tau_{Q \cdot \gamma} = \gamma^{-1}(\tau_Q)$ is contained in $\mathfrak F_\Gamma$. By definition, $Q$ is $\Gamma$-equivalent to $Q \cdot \gamma \in \Gamma\text{-}\mathrm{RF}(D)$. The uniqueness is an immediate consequence of (1).

(3) Let $\Gamma\text{-}\mathrm{RF}(D)$ and $\Gamma\text{-}\mathrm{RF}(D)'$ be the sets of $\Gamma$-reduced forms in $\mathrm{QF}(D)$ with respect to $\mathfrak F_\Gamma$ and $\mathfrak F_\Gamma'$, respectively. According to (2) there is a function $f:\Gamma\text{-}\mathrm{RF}(D) \rightarrow \Gamma\text{-}\mathrm{RF}(D)'$ by mapping $Q$ to a unique $Q'$ in $\Gamma\text{-}\mathrm{RF}(D)'$ which is $\Gamma$-equivalent to $Q$. Note that $f$ is injective by (1). For if $f(Q_1) = f(Q_2)$, then $Q_1, Q_2 \in \Gamma\text{-}\mathrm{RF}(D)$ are $\Gamma$-equivalent, so we have $Q_1 = Q_2$. Now the symmetry immediately implies our desired assertion.

(4) Let $\gamma_1, \ldots, \gamma_n$ be all the distinct left coset representatives of $\Gamma$ in $\mathrm{SL}_2(\mathbb Z)$, where $n = [\mathrm{SL}_2(\mathbb Z) : \Gamma]$. Then we can choose $\mathfrak F_\Gamma$ so that $\mathfrak F_\Gamma$ is contained in
\[ \bigcup_{i=1}^n \gamma_i^{-1}(\mathfrak F_{\mathrm{SL}_2(\mathbb Z)}). \]
Under this choice of $\mathfrak F_\Gamma$ we have
\begin{eqnarray*}
\Gamma\text{-}\mathrm{RF}(D) & = & \{ Q \in \mathrm{QF}(D) \, | \, \tau_Q \in \mathfrak F_\Gamma \} \\
& \subset & \{ Q \in \mathrm{QF}(D) \, | \, \gamma_i(\tau_Q) \in \mathfrak F_{\mathrm{SL}_2(\mathbb Z)} \mbox{ for some } i = 1, \ldots, n \} \\
& = & \{ Q \in \mathrm{QF}(D) \, | \, Q \cdot \gamma_i^{-1} \mbox{ is $\mathrm{SL}_2(\mathbb Z)$-reduced for some } i = 1, \ldots, n \} \\
& = & \bigcup_{i=1}^n \{ Q' \cdot \gamma_i \in \mathrm{QF}(D) \, | \, Q' \in \mathrm{SL}_2(\mathbb Z)\text{-}\mathrm{RF}(D) \}.
\end{eqnarray*}
Because $|\mathrm{SL}_2(\mathbb Z)\text{-}\mathrm{RF}(D)|$ is finite, we can derive that $|\Gamma\text{-}\mathrm{RF}(D)|$ is also finite.
\end{proof}

For a negative discriminant $D$ and a positive integer $N$ we define
\[ \mathcal C(D, \Gamma_0(N)) \ = \ \{ ax^2 + bxy + cy^2 \in \mathrm{QF}(D) \, | \, \gcd(a, N) = 1 \} / \sim \]
where $\sim$ denotes the $\Gamma_0(N)$-equivalence. We simply denote $\mathcal C(D, \Gamma_0(1))$ by $\mathcal C(D)$ in accordance with standard notation.

Given an imaginary quadratic field $K$, let $d_K$ denote the discriminant of $K$, and $\mathcal O_K = \mathbb Z + \mathbb Z \frac{d_K + \sqrt{d_K}}{2}$ the ring of integers in $K$, and $\mathcal O := \mathbb Z + f\mathcal O_K$ an order of conductor $f$ in $K$. Then $\mathcal O$ has discriminant $D := f^2 d_K$. Every fractional $\mathcal O$-ideal is of the form $\alpha \mathfrak a$ for some $\alpha \in K^\ast$ and some $\mathcal O$-ideal $\mathfrak a$, and vice versa. A fractional $\mathcal O$-ideal $\mathfrak a$ is said to be proper if it satisfies
\[
\mathcal O \ = \ \{ \alpha \in K \, | \, \alpha \mathfrak a \subset \mathfrak a  \}.
\]
Then it is easy to see that all principal $\mathcal O$-ideals and all $\mathcal O_K$-ideals are proper. An $\mathcal O$-ideal $\mathfrak a$ is said to be relatively prime to $N \in \mathbb N$ if $\mathfrak a + N\mathcal O = \mathcal O$. Then we see from the proof of \cite[Lemma 7.18]{Cox} that $\mathfrak a$ is relatively prime to $N$ if and only if its norm $\mathrm{N}(\mathfrak a) := |\mathcal O / \mathfrak a|$ is relatively prime to $N$, and that every $\mathcal O$-ideal relatively prime to $f$ is proper.

Let $\mathcal I(\mathcal O)$ denote the group of all proper fractional $\mathcal O$-ideals, and $\mathcal I(\mathcal O, N)$ the subgroup of $\mathcal I(\mathcal O)$ generated by all proper $\mathcal O$-ideals relatively prime to $N$. We also denote by $\mathcal P(\mathcal O, N)$ the subgroup of $\mathcal I(\mathcal O)$ generated by all principal $\mathcal O$-ideals that are relativel prime to $N$, and we denote by $\mathcal P_0(\mathcal O, N)$ the subgroup of $\mathcal I(\mathcal O)$ generated by all principal ideals $\alpha \mathcal O$ where $\alpha \in \mathcal O$, $\alpha \equiv a \ (\bmod \ N\mathcal O)$ for some $a \in \mathbb Z$ with $(a, N) = 1$. Then $\mathcal P_0(\mathcal O, N) \subset \mathcal P(\mathcal O, N) \subset \mathcal I(\mathcal O, N)$ and hence we may consider the quotient groups
\begin{eqnarray*}
\mathcal C(\mathcal O, N) & := & \mathcal I(\mathcal O, N)/\mathcal P(\mathcal O, N), \\
\mathcal C_0(\mathcal O, N) & := & \mathcal I(\mathcal O, N)/\mathcal P_0(\mathcal O, N).
\end{eqnarray*}
We employ the standard notation $\mathcal C(\mathcal O)$ and $\mathcal P(\mathcal O)$ instead of $\mathcal C(\mathcal O, 1)$ and $\mathcal P(\mathcal O, 1)$, respectively. It is trivial that $\mathcal I(\mathcal O) = \mathcal I(\mathcal O, 1)$, $\mathcal P(\mathcal O) = \mathcal P_0(\mathcal O, 1)$, and $\mathcal C(\mathcal O) = \mathcal C_0(\mathcal O, 1)$. Note also that $\mathcal I(\mathcal O_K, N)$ and $\mathcal P_0(\mathcal O_K, N)$ are equal to the standard notation $\mathcal I_K(N)$ and $\mathcal P_{K, \mathbb Z}(N)$ in Cox's book, respectively.

\begin{proposition}\label{Proposition - bijection}
If $\mathcal O$ is the order of discriminant $D$ in an imaginary quadratic field, then the map
\[
\Phi: \mathcal C(D, \Gamma_0(N)) \longrightarrow \mathcal C_0(\mathcal O, N)
\]
taking $[Q]$ to $[\mathbb Z + \mathbb Z \tau_Q]$ is a bijection.
\end{proposition}

\begin{proof}
(well-defined) Let $Q(x, y) = ax^2 + bxy + cy^2 \in \mathrm{QF}(D)$ satisfy $(a, N) = 1$. Then $\mathcal O = \mathbb Z + \mathbb Z a \tau_Q$ and $\mathbb Za + \mathbb Z a\tau_Q$ is a proper $\mathcal O$-ideal. Because its norm
\[ \mathrm{N}(\mathbb Za + \mathbb Z a\tau_Q) \ = \ |\mathcal O / (\mathbb Za + \mathbb Z a \tau_Q)| = a \]
is relatively prime to $N$, we see that $\mathbb Z + \mathbb Z\tau_Q \in \mathcal I(\mathcal O, N)$.


Now assume that $Q(x, y)$ and $Q'(x, y)$ are $\Gamma_0(N)$-equivalent to each other. Then we have $\tau_{Q'} = \gamma (\tau_Q)$ for some $\gamma = ( \begin{smallmatrix} a' & b' \\ c' & d' \end{smallmatrix} ) \in \Gamma_0(N)$. Note that
\[
(c'\tau_Q + d')(\mathbb Z + \mathbb Z \tau_{Q'}) \ = \ \mathbb Z(c'\tau_Q + d') + \mathbb Z(a'\tau_Q + b') \ = \ \mathbb Z + \mathbb Z \tau_Q
\]
and that 
\[ (c'\tau_Q + d')\mathcal O \ = \ (ac'\tau_Q + ad')a^{-1}\mathcal O \ \in \ \mathcal P_0(\mathcal O, N) \]
because $c'a\tau_Q + ad' \in \mathcal O$, $c'a\tau_Q + ad' \equiv a \pmod{N\mathcal O}$, and $(a, N) = 1$. Thus we have $[\mathbb Z + \mathbb Z \tau_Q] = [\mathbb Z + \mathbb Z \tau_{Q'}]$ in $\mathcal C_0(\mathcal O, N)$.

(1-to-1) Assume that $[\mathbb Z + \mathbb Z \tau_Q] = [\mathbb Z + \mathbb Z \tau_{Q'}]$, i.e. that $\alpha(\mathbb Z + \mathbb Z \tau_{Q'}) = \mathbb Z + \mathbb Z \tau_Q$ for some $\alpha \mathcal O \in \mathcal P_0(\mathcal O, N)$. Then we see that $\alpha \tau_{Q'} = a \tau_Q + b$ and $\alpha = c\tau_Q + d$ for some $a, b, c, d \in \mathbb Z$, from which we obtain that $\tau_{Q'} = \gamma(\tau_Q)$ with $\gamma := (\begin{smallmatrix} a & b \\ c & d \end{smallmatrix}) \in M_2(\mathbb Z)$. Since
\[ \mathbb Z + \mathbb Z \tau_Q \ = \ \alpha(\mathbb Z + \mathbb Z \tau_{Q'}) \ = \ \mathbb Z(c\tau_Q + d) + \mathbb Z(a\tau_Q + b) \]
and $\tau_Q, \tau_{Q'} \in \mathbb H$, we derive that $\gamma \in SL_2(\mathbb Z)$.

Now we show that $\gamma \in \Gamma_0(N)$. By definition, we may set $\alpha = \beta \delta^{-1}$, where $\beta, \delta \in \mathcal O$ and $\beta \equiv b' \pmod {N\mathcal O}$, $\delta \equiv d' \pmod {N\mathcal O}$ and $b', d' \in \mathbb Z$, $(b'd', N) = 1$. Let $a' = Q(1, 0)$. Then $\mathcal O = \mathbb Z + \mathbb Z a'\tau_Q$. Since $a'\alpha = c a' \tau_Q + a'd \in \mathcal O$, we have
\begin{eqnarray*}
a'\alpha \delta & \equiv & c d' a' \tau_Q + a'dd' \pmod{N\mathcal O},\\
a'\alpha \delta & = & a'\beta \ \equiv \ a'b' \pmod{N\mathcal O},
\end{eqnarray*}
and hence $c \equiv 0 \pmod N$. This says that $Q(x, y)$ is $\Gamma_0(N)$-equivalent to $Q'(x, y)$.

(onto) Let $\mathfrak a \in \mathcal I(\mathcal O, N)$. Then we may write $\mathfrak a = s^{-1} \mathfrak b$, where $\mathfrak b$ is a proper $\mathcal O$-ideal relatively prime to $N$ and $s$ is an integer relatively prime to $N$. Let $t = \min \{ k \in \mathbb N \, | \, k \in \mathfrak b \}$. Since $\mathfrak b \bar{\mathfrak b} = \mathrm{N}(\mathfrak b) \mathcal O$, we have $\mathrm{N}(\mathfrak b) \in \mathfrak b$ and so $t \leq \mathrm{N}(\mathfrak b)$. We write $\mathrm{N}(\mathfrak b) = tq + r$ for some $0 \leq r < t$. Since $r = \mathrm{N}(\mathfrak b) - tq \in \mathfrak b$, we deduce that $r = 0$. Thus $t$ divides $\mathrm{N}(\mathfrak b)$ and hence $(t, N) = 1$.

Write $\mathfrak b = \mathbb Z \alpha + \mathbb Z \beta$ for some $\alpha, \beta \in \mathcal O$ and put $t = u \alpha + v \beta$ for some $u, v \in \mathbb Z$. Notice that $(u, v) = 1$ by the minimality of $t$. So there exist $w, z \in \mathbb Z$ such that $uz - vw = 1$. Since $(\begin{smallmatrix} u & v \\ w & z \end{smallmatrix}) \in \mathrm{SL}_2(\mathbb Z)$, we have
\[ \mathfrak b \ = \ \mathbb Z \alpha + \mathbb Z \beta \ = \ \mathbb Z(u\alpha + v\beta) + \mathbb Z(w\alpha + z\beta) \ = \ t(\mathbb Z + \mathbb Z \tau), \]
where $\tau = \frac{w\alpha + z\beta}{t}$. Replacing $\tau$ by $-\tau$ if necessary, we may assume that $\tau \in \mathbb H$. Let $ax^2 + bx + c \in \mathbb Z[x]$ be the minimal polynomial having $\tau$ as its zero, where $(a, b, c) = 1$ and $a > 0$. Clearly, $Q(x, y) := ax^2 + bxy + cy^2$ is a primitive positive-definite quadratic form. Because $\mathbb Z + \mathbb Z\tau$ is a proper fractional ideal of both $\mathcal O$ and $\mathbb Z + \mathbb Z a\tau$, we see that $\mathcal O = \mathbb Z + \mathbb Z a\tau$ and hence that $Q$ has discriminant $D$. Moreover, we deduce that $(a, N) = 1$ because $t \tau \in \mathfrak b \subset \mathcal O = \mathbb Z + \mathbb Z a\tau$ implies that $a$ divides $t$. Finally, since $st^{-1}\mathcal O \in \mathcal P_0(\mathcal O, N)$, we have
\[
\Phi([Q(x, y)]) \ = \ [\mathbb Z + \mathbb Z \tau] \ = \ [ st^{-1} \mathfrak a] \ = \ [\mathfrak a].
\]
This completes the proof.
\end{proof}

\begin{lemma}\label{Lemma - relatively prime for quadratic forms}
For any $Q \in \mathrm{QF}(D)$ and $M, N \in \mathbb N$ with $(M, N) = 1$, there exists $Q' \in \mathrm{QF}(D)$ such that $Q'$ is $\Gamma_0(N)$-equivalent to $Q$ and $Q'(1, 0)$ is relatively prime to $M$.
\end{lemma}

\begin{proof}
For now, let us assume that for any prime divisor $p_k$ $(1 \leq k \leq s)$ of $M$ there is a pair $(a_k, c_k)$ of integers such that $\mathrm{gcd}(a_k, c_k) = 1$ and $p_k \nmid Q(a_k, c_k)$. Applying the Chinese remainder theorem, we can find $c \in \mathbb Z$ such that $c \equiv 0 \pmod N$ and $c \equiv c_k \pmod{p_k}$ for all $k = 1, \ldots, s$. Now, relabeling the indices if necessary, we may assume that $p_i \nmid a_i$ for $1 \leq i \leq r$ and that $p_j|a_j$ for $r < j \leq s$. Clearly, $p_j \nmid c_j$ for $r < j \leq s$. According to both the Chinese remainder theorem and the Dirichlet theorem on primes in arithmetic progressions, there is a prime $q$ such that $q \equiv [p_{r+1} \cdots p_s]_{p_i}^{-1} a_i \pmod {p_i}$ for all $i = 1, \ldots, r$ and $q \nmid c$.

Set $a = p_{r+1} \cdots p_s q$. Note that $\mathrm{gcd}(a, c) = 1$ because $q \nmid c$ and $\mathrm{gcd}(p_j, c) = \mathrm{gcd}(p_j, c_j) = 1$ for $r < j \leq s$. Moreover, we see that $\mathrm{gcd}(M, Q(a, c)) = 1$ because $Q(a, c) \equiv Q(a_i, c_i) \not \equiv 0 \pmod{p_i}$ for $1 \leq i \leq r$ and $Q(a, c) \equiv Q(0, c_j) \equiv Q(a_j, c_j) \not \equiv 0 \pmod {p_j}$ for $r < j \leq s$. Since $\mathrm{gcd}(a, c) = 1$, we have $\gamma := (\begin{smallmatrix} a & b \\ c & d \end{smallmatrix}) \in \Gamma_0(N)$ for some $b, d \in \mathbb Z$. Put $Q' = Q \cdot \gamma$. Then we have shown that $Q'(1, 0) = Q(a, c)$ is relatively prime to $M$.

Finally, it remains to show that for any prime $p$ there is a pair $(a_0, c_0)$ of integers such that $\mathrm{gcd}(a_0, c_0) = 1$ and $p \nmid Q(a_0, c_0)$. Since $Q$ is primitive, it is readily seen that $\mathrm{gcd}(Q(1, 0),$ $Q(1, 1),$ $Q(0, 1)) = 1$, from which it follows directly that at least one of $Q(1, 0)$, $Q(1, 1)$, and $Q(0, 1)$ is not divisible by $p$.
\end{proof}

We now review the formula for the Dirichlet composition. Given two forms $Q(x, y) = ax^2 + bxy + cy^2$, $Q'(x, y) = a'x^2 +b'xy + c'y^2$ $\in$ $\mathrm{QF}(D)$ satisfying $\gcd(a, a', \frac{b+b'}{2}) = 1$, there exists a unique integer $B$ modulo $2aa'$ such that
\begin{eqnarray*}
B & \equiv & b \pmod {2a} \\
B & \equiv & b' \pmod {2a'} \\
B^2 & \equiv & D \pmod {4aa'}.
\end{eqnarray*}
Then the Dirichlet composition of $Q(x, y)$ and $Q'(x, y)$ is defined to be the form
\[ Q''(x, y) \ = \ aa' x^2 + Bxy + \frac{B^2 - D}{4aa'} y^2 \ \in \ \mathrm{QF}(D). \]

We will show that the Dirichlet composition makes $\mathcal C(D, \Gamma_0(N))$ an abelian group. Let $[Q(x, y)]$, $[Q'(x, y)]$ $\in$ $\mathcal C(D, \Gamma_0(N))$ be given as $Q(x, y) = ax^2 + bxy + cy^2$, $Q'(x, y) = a'x^2 + b'xy + c'y^2$. According to the previous lemma we may assume that $(a, a') = 1$. Then the binary operation on $\mathcal C(D, \Gamma_0(N))$ is defined by
\[ [Q(x, y)] \cdot [Q'(x, y)] \ = \ [Q''(x, y)], \]
where $Q''(x, y)$ is the Dirichlet composition of $Q(x, y)$ and $Q'(x, y)$.

\begin{theorem}\label{Theorem - isomorphism}
The set $\mathcal C(D, \Gamma_0(N))$ is an abelian group under the operation defined as above and the map $\Phi: \mathcal \mathcal C(D, \Gamma_0(N)) \rightarrow \mathcal C_0(\mathcal O, N)$ in Proposition \ref{Proposition - bijection} is an isomorphism.
\end{theorem}

\begin{proof}
We need to show that the operation induced from the Dirichlet composition is well defined. Let $Q$, $Q'$, and $Q''$ be given as above. First of all, the operation is independent of the choice of $B$. For, if $B' \equiv B \pmod {2aa'}$, then $aa' x^2 + B'xy + \frac{(B')^2 - D}{4aa'} y^2$ is $\Gamma_0(N)$-equivalent to $Q''$ because $\Gamma_0(N)$ contains $(\begin{smallmatrix} 1 & 1 \\ 0 & 1 \end{smallmatrix})$.

Now we consider the three proper $\mathcal O$-ideals
\[ \mathbb Za + \mathbb Z\frac{-b + \sqrt{D}}{2}, \ \mathbb Za' + \mathbb Z\frac{-b' + \sqrt{D}}{2}, \ \mathbb Zaa' + \mathbb Z\frac{-B + \sqrt{D}}{2}, \]
which are associated with $Q$, $Q'$, and $Q''$, respectively. We claim that
\[ (\mathbb Za + \mathbb Z\frac{-b + \sqrt{D}}{2}) \cdot (\mathbb Za' + \mathbb Z\frac{-b' + \sqrt{D}}{2}) \ = \ \mathbb Zaa' + \mathbb Z\frac{-B + \sqrt{D}}{2} \]
in $\mathcal I(\mathcal O, N)$. Setting $\Delta = \frac{-B + \sqrt{D}}{2}$ and making use 
of $B \equiv b \pmod{2a}$ and $B \equiv b' \pmod{2a'}$, the above three ideals can be written as $\mathbb Za + \mathbb Z\Delta$, $\mathbb Za' + \mathbb Z\Delta$, and $\mathbb Zaa' + \mathbb Z\Delta$, respectively. Since $\Delta^2 \equiv -B\Delta \ (\bmod \ {aa'})$ by the congruence equation $B^2 \equiv D \ (\bmod \ {4aa'})$, we have
\begin{eqnarray*}
(\mathbb Za + \mathbb Z\Delta) \cdot (\mathbb Za' + \mathbb Z\Delta) & = & \mathbb Z aa' + \mathbb Z a \Delta + \mathbb Za'\Delta + \mathbb Z\Delta^2 \\
& = & \mathbb Z aa' + \Delta(\mathbb Za + \mathbb Za' + \mathbb ZB) \\
& = & \mathbb Zaa' + \mathbb Z\Delta.
\end{eqnarray*}
Here we have used the property that $(a, a', B) = 1$.

Let $[Q] = [Q_0]$ and $[Q'] = [Q'_0]$, where $Q_0 = a_0x^2 + b_0xy + c_0y^2$ and $Q'_0 = a'_0x^2 + b'_0xy + c'_0y^2$ with $(a_0, a'_0, \frac{b_0 + b_0'}{2}) = 1$. Let $Q''_0 = a_0a'_0x^2 + B_0xy + \frac{B_0^2 - D}{4a_0a'_0}y^2$ be the Dirichlet composition of $Q_0$ and $Q'_0$. To complete the proof of the well-definedness, it suffices to show that $\Phi([Q'']) = \Phi([Q''_0])$ as $\Phi$ is bijective. Note that
\begin{eqnarray*}
\Phi([Q'']) & = & [\mathbb Z + \mathbb Z \tau_{Q''}] \ = \ \lbrack \mathbb Zaa' + \mathbb Z\Delta \rbrack \ = \ \lbrack \mathbb Za + \mathbb Z\Delta \rbrack \cdot \lbrack \mathbb Za' + \mathbb Z \Delta \rbrack \\
& = & \Phi([Q]) \cdot \Phi([Q']) \ = \ \Phi([Q_0]) \cdot \Phi([Q'_0]) \\
& = & \lbrack \mathbb Z a_0 + \mathbb Z\frac{-b_0 + \sqrt{D}}{2} \rbrack \cdot \lbrack \mathbb Z a'_0 + \mathbb Z\frac{-b'_0 + \sqrt{D}}{2} \rbrack \\
& = & \lbrack \mathbb Z a_0 a'_0 + \mathbb Z \frac{-B_0 + \sqrt{D}}{2} \rbrack \\
& = & \Phi([Q''_0]).
\end{eqnarray*}
Here we have employed the claim above. Therefore the operation is well defined and automatically it has been proved that $\mathcal C(D, \Gamma_0(N))$ is an abelian group and that $\Phi$ is an isomorphism.
\end{proof}

\begin{corollary}
The identity element of $\mathcal C(D, \Gamma_0(N))$ is
\[
\begin{array}{ll}
\lbrack x^2 - \frac{D}{4} y^2 \rbrack & \mbox{if } D \equiv 0 \ (\bmod \ 4), \\
\lbrack x^2 + xy + \frac{1 - D}{4} y^2 \rbrack & \mbox{if } D \equiv 1 \ (\bmod \ 4),
\end{array}
\]
and the inverse element of $[ax^2 + bxy + cy^2]$ in $\mathcal C(D, \Gamma_0(N))$ is $[ax^2 - bxy + cy^2]$.
\end{corollary}

\begin{proof}
It follows immediately from the facts that
\[
\begin{array}{ll}
\Phi(\lbrack x^2 - \frac{D}{4} y^2 \rbrack) \ = \ \lbrack \mathbb Z + \mathbb Z \frac{\sqrt{D}}{2} \rbrack \ = \ \lbrack \mathcal O \rbrack & \mbox{if } D \equiv 0 \ (\bmod \ 4), \\
\Phi(\lbrack x^2 + xy + \frac{1 - D}{4} y^2 \rbrack) \ = \ \lbrack \mathbb Z + \mathbb Z \frac{-1 + \sqrt{D}}{2} \rbrack \ = \ \lbrack \mathcal O \rbrack \quad & \mbox{if } D \equiv 1 \ (\bmod \ 4),
\end{array}
\]
and that
\begin{eqnarray*}
& & \hspace{-.6cm} \Phi(\lbrack ax^2 + bxy + cy^2 \rbrack) \cdot \Phi(\lbrack ax^2 - bxy + cy^2 \rbrack) \\
& = & \lbrack \mathbb Za + \mathbb Z\frac{-b + \sqrt{D}}{2} \rbrack \cdot \lbrack \mathbb Za + \mathbb Z\frac{b + \sqrt{D}}{2} \rbrack \\
& = & \lbrack \mathbb Za^2 + \mathbb Z ab + \mathbb Z ac + \mathbb Z \frac{-ab + a\sqrt{D}}{2} \rbrack \\
& = & \lbrack a \cdot (\mathbb Z + \mathbb Z\frac{-b + \sqrt{D}}{2}) \rbrack \\
& = & [ \mathcal O ].
\end{eqnarray*}
\end{proof}

The following proposition is a slight generalization of Proposition 7.20 in Cox's book \cite{Cox} and can be proved by the same argument as in the proof of Cox's. However, we present here the proof for the sake of completeness.

\begin{proposition}\label{Proposition - norm-preserving isomorphism}
Let $\mathcal O$ be an order of conductor $f$ in an imaginary quadratic field $K$ and let $N$ denote a positive integer. Then the map $\mathfrak a$ $\longmapsto$ $\mathfrak a \cap \mathcal O$ induces a norm-preserving isomorphism from $\mathcal I(\mathcal O_K, Nf)$ to $\mathcal I(\mathcal O, Nf)$ and its inverse map is given by $\mathfrak b$ $\longmapsto$ $\mathfrak b \mathcal O_K$.
\end{proposition}

\begin{proof}
Let $\mathfrak a$ be an $\mathcal O_K$-ideal relatively prime to $Nf$. Then the natural map
\[ \mathcal O / (\mathfrak a \cap \mathcal O) \longrightarrow \mathcal O_K / \mathfrak a \]
is an injective homomorphism. Moreover, this map is also surjective since $\mathfrak a + Nf\mathcal O_K = \mathcal O_K$ and $f \mathcal O_K \subset \mathcal O$. Thus we have $\mathrm{N}(\mathfrak a) = \mathrm{N}(\mathfrak a \cap \mathcal O)$ and so $\mathfrak a \cap \mathcal O \in \mathcal I(\mathcal O, Nf)$. We further deduce that $(\mathfrak a \cap \mathcal O)\mathcal O_K = \mathfrak a$ because $(\mathfrak a \cap \mathcal O)\mathcal O_K \subset \mathfrak a$ is clear and
\[ \mathfrak a \ = \ \mathfrak a \mathcal O \ = \ \mathfrak a(\mathfrak a \cap \mathcal O + Nf\mathcal O) \ \subset \ (\mathfrak a \cap \mathcal O)\mathcal O_K + f \mathfrak a \ \subset \ (\mathfrak a \cap \mathcal O)\mathcal O_K. \]

Now let $\mathfrak b$ be an $\mathcal O$-ideal relatively prime to $Nf$. Then we see
\[ \mathfrak b \mathcal O_K + Nf \mathcal O_K \ = \ (\mathfrak b + Nf\mathcal O)\mathcal O_K \ = \ \mathcal O \mathcal O_K \ = \ \mathcal O_K, \]
whence $\mathfrak b\mathcal O_K$ is an $\mathcal O_K$-ideal relatively prime to $Nf$. We further deduce that $\mathfrak b \mathcal O_K \cap \mathcal O = \mathfrak b$ because $\mathfrak b \subset \mathfrak b \mathcal O_K \cap \mathcal O$ is clear and
\begin{eqnarray*}
\mathfrak b \mathcal O_K \cap \mathcal O & = & (\mathfrak b \mathcal O_K \cap \mathcal O) \mathcal O \ = \ (\mathfrak b \mathcal O_K \cap \mathcal O)(\mathfrak b + Nf\mathcal O) \\
& \subset & \mathfrak b + Nf(\mathfrak b\mathcal O_K \cap \mathcal O) \ \subset \ \mathfrak b + Nf \mathfrak b \mathcal O_K \\
& \subset & \mathfrak b + N \mathfrak b \ \subset \ \mathfrak b.
\end{eqnarray*}

This shows that we get a bijection between the monoids of $\mathcal O_K$-ideals and $\mathcal O$-ideals that are relatively prime to $Nf$. Since the inverse map is obviously multiplicative, the bijection induces a well-defined isomorphism between $I(\mathcal O_K, Nf)$ and $I(\mathcal O, Nf)$.
\end{proof}

\begin{lemma}\label{Lemma - relatively prime for ideals}
Let $\mathcal O$ be an order in an imaginary quadratic field and let $M, N$ be positive integers which are relatively prime. Then every ideal class in $\mathcal C_0(\mathcal O, N)$ contains a proper $\mathcal O$-ideal whose norm is relatively prime to $M$.
\end{lemma}

\begin{proof}
Since $\Phi$ is bijective, every ideal class in $\mathcal C_0(\mathcal O, N)$ is of the form $[\mathbb Z + \mathbb Z \tau_Q]$ for some $Q = ax^2 + bxy + cy^2 \in \mathrm{QF}(D)$ with $(a, N) = 1$, where $D$ is the discriminant of $\mathcal O$. Appealing to Lemma \ref{Lemma - relatively prime for quadratic forms}, we may assume that $a$ is also relatively prime to $M$. Since $[\mathbb Z + \mathbb Z \tau_Q] = [\mathbb Za + \mathbb Za \tau_Q]$, the desired assertion readily follows from the proof of Proposition \ref{Proposition - bijection}, which says that $\mathbb Za + \mathbb Za\tau_Q$ is a proper $\mathcal O$-ideal having norm $a$.
\end{proof}

\begin{theorem}
Let $\mathcal O$ be an order of conductor $f$ in an imaginary quadratic field $K$ and let $N$ denote a positive integer. Then the map
\[ \Psi : \mathcal C_0(\mathcal O_K, Nf) \longrightarrow \mathcal C_0(\mathcal O, N) \]
given by $[\mathfrak a \mathfrak b^{-1}] \longmapsto [(\mathfrak a \cap \mathcal O)(\mathfrak b \cap \mathcal O)^{-1}]$ is an isomorphism, where $\mathfrak a$ and $\mathfrak b$ are $\mathcal O_K$-ideals relatively prime to $Nf$.
\end{theorem}

\begin{proof}
We consider the homomorphism
\[ \Psi' : \mathcal I(\mathcal O_K, Nf) \longrightarrow \mathcal C_0(\mathcal O, N) \]
defined multiplicatively from the map taking $\mathfrak a$ to $[\mathfrak a \cap \mathcal O]$, where $\mathfrak a$ is an $\mathcal O_K$-ideal relatively prime to $Nf$. According to Proposition \ref{Proposition - norm-preserving isomorphism}, this map is well defined because $\mathfrak a \cap \mathcal O \in \mathcal I(\mathcal O, Nf) \subset \mathcal I(\mathcal O, N)$. The map is also surjective because Lemma \ref{Lemma - relatively prime for ideals} says that every ideal class in $\mathcal C_0(\mathcal O, N)$ contains a proper $\mathcal O$-ideal $\mathfrak b$ which is relatively prime to $Nf$ and because we have $\mathfrak b \mathcal O_K \cap \mathcal O = \mathfrak b$ by applying Proposition \ref{Proposition - norm-preserving isomorphism} again.

We claim that $\mathrm{ker}(\Psi') = \mathcal P_0(\mathcal O_K, Nf)$.

($\supset$) Let $\alpha = a + Nf \beta$ where $\beta \in \mathcal O_K$, $a \in \mathbb Z$, and $(a, Nf) = 1$. Since $f\beta \in \mathcal O$ and $(a, N) = 1$, we infer from Proposition \ref{Proposition - norm-preserving isomorphism} that
\[ \alpha \mathcal O_K \cap \mathcal O \ = \ (\alpha \mathcal O)\mathcal O_K \cap \mathcal O \ = \ \alpha \mathcal O \ \in \ \mathcal P_0(\mathcal O, N). \]

($\subset$) Let $\mathfrak a \mathfrak b^{-1} \in \mathrm{ker}(\Psi')$ where $\mathfrak a$, $\mathfrak b$ are $\mathcal O_K$-ideals relatively prime to $Nf$. Since $\mathrm{N}(\mathfrak b) \mathcal O_K = \mathfrak b \bar{\mathfrak b}$, we have $\mathfrak a \mathfrak b^{-1} = \mathrm{N}(\mathfrak b)^{-1} \mathfrak a \bar{\mathfrak b}$ and hence we see $\mathfrak a \bar{\mathfrak b} \in \mathrm{ker}(\Psi')$ and it is sufficient to show that $\mathfrak a \bar{\mathfrak b} \in \mathcal P_0(\mathcal O_K, Nf)$. Set $\mathfrak a \bar{\mathfrak b} \cap \mathcal O = \alpha \mathcal O$ where $\alpha = a + N\beta$, $\beta \in \mathcal O$, $a \in \mathbb Z$, and $(a, N) = 1$. Since $\mathcal O = \mathbb Z + f \mathcal O_K$, we may put $\alpha = b + Nf \gamma$ where $\gamma \in \mathcal O_K$, $b \in \mathbb Z$, and $(b, N) = 1$. Then we deduce that $(b, Nf) = 1$ because $(\mathrm{N}(\mathfrak a \bar{\mathfrak b}), Nf) = 1$ and
\[ \mathrm{N}(\mathfrak a \bar{\mathfrak b}) \ = \ \mathrm{N}(\mathfrak a \bar{\mathfrak b} \cap \mathcal O) \ = \ \mathrm{N}(\alpha \mathcal O) \ = \ \mathrm{N}(\alpha) \ = \ b^2 + Nf(\gamma + \bar{\gamma}) + N^2 f^2 \gamma \bar{\gamma}. \]
Therefore we obtain by applying Proposition \ref{Proposition - norm-preserving isomorphism} once more that
\[ \mathfrak a \bar{\mathfrak b} \ = \ (\mathfrak a \bar{\mathfrak b} \cap \mathcal O) \mathcal O_K \ = \ \alpha \mathcal O \mathcal O_K \ = \ \alpha \mathcal O_K \ \in \ \mathcal P_0(\mathcal O_K, Nf). \]
\end{proof}

\begin{corollary}
(1) Let $f, N$ be positive integers and let $\mathcal O$ (respectively, $\mathcal O'$) be the order of conductor $f$ (respectively, $Nf$) in an imaginary quadratic field $K$. Then we have
\[ \mathcal C_0(\mathcal O, N) \ \cong \ \mathcal C(\mathcal O'). \]

(2) Let $N$ be a positive integer and let $D \equiv 0, 1 \pmod 4$ be a negative integer. Then we have
\[ \mathcal C(D, \Gamma_0(N)) \ \cong \ \mathcal C(DN^2). \]
\end{corollary}

It is also true that $\mathcal C(\mathcal O') \cong \mathcal C(\mathcal O', Nf)$ according to \cite[Proposition 7.19]{Cox}.

\begin{proof}
(1) By the above theorem, we readily see that
\[ \mathcal C_0(\mathcal O, N) \ \cong \ \mathcal C_0(\mathcal O_K, Nf) \ \cong \ \mathcal C_0(\mathcal O', 1) \ = \ \mathcal C(\mathcal O'). \]

(2) It follows from (1) and Theorem \ref{Theorem - isomorphism}.
\end{proof}

\section{The $N$-representations of integers by binary quadratic forms}

\begin{definition}
Let $f(x, y) \in \mathbb Z[x, y]$ be a binary quadratic form and let $N \in \mathbb N$. An integer $m$ is said to be $N$-represented by $f(x, y)$ if $m = f(x_0, y_0)$ for some $x_0, y_0 \in \mathbb Z$ with $(x_0, N) = 1$, $y_0 \equiv 0 \ (\bmod \ N)$. If $(x_0, y_0) = 1$ is further satisfied, then we say that $m$ is properly $N$-represented by $f(x, y)$.
\end{definition}

\begin{remark}
If $f(x, y)$ and $g(x, y)$ are $\Gamma_0(N)$-equivalent, then
\[
\{ f(x, y) \in \mathbb Z \, | \, (x, N) = 1, \, y \equiv 0 \, (\bmod \, N) \} \, = \, \{ g(x, y) \in \mathbb Z \, | \, (x, N) = 1, \, y \equiv 0 \, (\bmod \, N) \}.
\]
\end{remark}

\begin{lemma}
An integer $m$ is properly $N$-represented by $f(x, y)$ if and only if there exist $b, c \in \mathbb Z$ such that $f(x, y)$ is $\Gamma_0(N)$-equivalent to $mx^2 + bxy + cy^2$.
\end{lemma}

\begin{proof}
($\Leftarrow$) By definition, there exists $(\begin{smallmatrix} a & b \\ c & d \end{smallmatrix}) \in \Gamma_0(N)$ such that $mx^2 + bxy + cy^2 = f(ax + by, cx + dy)$. Substituting $x = 1$ and $y = 0$, we see that $m = f(a, c)$. Since $(a, N) = 1$, $c \equiv 0 \ (\bmod \ N)$, and $(a, c) = 1$, $m$ is properly $N$-represented by $f(x, y)$.

($\Rightarrow$) By definition, there exist $x_0, y_0 \in \mathbb Z$ such that $m = f(x_0, y_0)$, $(x_0, N) = 1$, $y_0 \equiv 0 \ (\bmod \ N)$ and $(x_0, y_0) = 1$. Since $(x_0, y_0) = 1$, we see that $ux_0 + vy_0 = 1$ for some $u, v \in \mathbb Z$. Put $\gamma = (\begin{smallmatrix} x_0 & -v \\ y_0 & u \end{smallmatrix}) \in \Gamma_0(N)$. Then $f(x, y)$ is $\Gamma_0(N)$-equivalent to $f(x_0x - vy, y_0x + uy)$. Because the coefficient of $x^2$ in $f(x_0x - vy, y_0x + uy)$ is $f(x_0, y_0) = m$, we are done.
\end{proof}

\begin{lemma}
Let $D \equiv 0, 1 \ (\bmod \ 4)$ be an integer, $m$ an odd integer with $(m, D) = 1$, and $N \in \mathbb N$. Then $m$ is properly $N$-represented by a primitive form of discriminant $D$ if and only if $D$ is a quadratic residue modulo $m$.
\end{lemma}

\begin{proof}
($\Rightarrow$) By the preceding lemma, we may assume that $m$ is properly $N$-represented by a form $mx^2 + bxy + cy^2$ of discriminant $D$. Since $D = b^2 - 4mc$, $D$ is a quadratic residue modulo $m$.

($\Leftarrow$) Suppose $D \equiv b^2 \ (\bmod \ m)$ for some $b \in \mathbb Z$. Since $m$ is odd, we may assume that $D$ and $b$ have the same parity. From $D \equiv 0, 1 \ (\bmod \ 4)$, we infer that $D \equiv b^2 \ (\bmod \ 4m)$, i.e. that $D = b^2 - 4mc$ for some $c \in \mathbb Z$. Then $m$ is properly $N$-represented by the form $mx^2 + bxy + cy^2$ of discriminant $D$. Note that $mx^2 + bxy + cy^2$ is primitive because $(m, D) = 1$.
\end{proof}

\begin{corollary}
Given a positive integer $m$, let $p$ denote an odd prime not dividing $m$. Then for any $N \in \mathbb N$, we obtain that $(\frac{-m}{p}) = 1$ if and only if $p$ is $N$-represented by a primitive form of discriminant $-4m$.
\end{corollary}

\begin{proof}
This follows immediately from the preceding lemma because $-4m$ is a quadratic residue modulo $p$ if and only if $(\frac{-4m}{p}) = (\frac{-m}{p}) = 1$.
\end{proof}

\begin{proposition}
Given a positive integer $m$, let $p$ denote an odd prime not dividing $m$. Then for any $N \in \mathbb N$, we obtain that $(\frac{-m}{p}) = 1$ if and only if $p$ is $N$-represented by one of the $\Gamma_0(N)$-reduced forms of discriminant $-4m$.
\end{proposition}

\begin{proof}
The proof is clear.
\end{proof}

\begin{theorem}
Let $D \equiv 0, 1 \ (\bmod \ 4)$ be a negative integer, $N \in \mathbb N$, and let $\chi:(\mathbb Z / D\mathbb Z)^\times \rightarrow \{ \pm 1 \}$ be a homomorphism given by $\chi([m]) = (\frac{D}{m})$, where $(\frac{\cdot}{\cdot})$ denotes the Kronecker symbol. Then for any odd prime $p$ not dividing $D$, the residue class $[p] \in \mathrm{ker}(\chi)$ if and only if $p$ is $N$-represented by one of the $\Gamma_0(N)$-reduced forms of discriminant $D$.
\end{theorem}

\begin{proof}
Notice that
\begin{eqnarray*}
& & [p] \in \mathrm{ker}(\chi) \\
& \Leftrightarrow & (\frac{D}{p}) = 1 \\
& \Leftrightarrow & \mbox{$D$ is a quadratic residue modulo $p$} \\
& \Leftrightarrow & \mbox{$p$ is $N$-represented by a primitive form of discriminant $D$} \\
& \Leftrightarrow & \mbox{$p$ is $N$-represented by one of the $\Gamma_0(N)$-reduced forms of discriminant $D$.}
\end{eqnarray*}
\end{proof}

\begin{definition}
Given a negative integer $D \equiv 0, 1 \ (\bmod \ 4)$, the principal form is defined to be
\[
\left\{
\begin{array}{ll}
x^2 - \frac{D}{4}y^2 & \mbox{if } D \equiv 0 \ (\bmod \ 4) \\
x^2 + xy + \frac{1 - D}{4}y^2 & \mbox{if } D \equiv 1 \ (\bmod \ 4).
\end{array}
\right.
\]
\end{definition}

\begin{lemma}
Let $f(x, y) = ax^2 + bxy + cy^2$ be a primitive form and let $n, N \in \mathbb N$ with $(a, N) = 1$. Then there exists $m \in \mathbb N$ such that $(m, nN) = 1$ and $m$ is properly $N$-represented by $f(x, y)$.
\end{lemma}

\begin{proof}
Given a prime $p$ not dividing $N$, at least one of $f(1, 0)$, $f(1, N)$, $f(1 + N, N)$ is not divisible by $p$. Hence for each prime divisor $p_i$ ($i = 1, \ldots, \ell$) of $n$ not dividing $N$, we deduce by relabeling the indices if necessary that $p_i$ does not divide $f(x_i, y_i)$ for all $i$, where
\begin{eqnarray*}
x_i = 1 \mbox{ or } 1 + N, & & y_i = N \qquad (1 \leq i \leq k), \\
x_i = 1, & & y_i = 0 \qquad (k + 1 \leq i \leq \ell).
\end{eqnarray*}
By the Chinese remainder theorem, there exists $x \in \mathbb Z$ such that
\[
\left\{
\begin{array}{ll}
x \equiv x_i & (\bmod \ {p_i}) \qquad (1 \leq i \leq \ell), \\
x \equiv 1 & (\bmod \ N).
\end{array}
\right.
\]
By Dirichlet's theorem on arithmetic progressions, we can find a prime $q$ such that $(q, x) = 1$ and $p_{k + 1} \cdots p_\ell q \equiv 1 \ (\bmod \ {p_1 \cdots p_k})$.

Put $y = N p_{k + 1} \cdots p_\ell q$ and $m = f(x, y)$. Then
\begin{eqnarray*}
n \equiv f(x_i, y_i) \not \equiv 0 \ (\bmod \ {p_i}) & & (1 \leq i \leq \ell), \\
n \equiv f(1, 0) = a \ (\bmod \ N), & & 
\end{eqnarray*}
whence $(m, nN) = 1$. Moreover, we have $x \equiv 1 \ (\bmod \ N)$, $y \equiv 0 \ (\bmod \ N)$, and $(x, y) = 1$ because
\[ (x, p_i) \ = \ (x_i, p_i) \ = \ 1 \qquad (k + 1 \leq i \leq \ell). \]
\end{proof}

\begin{lemma}
Given a negative integer $D \equiv 0, 1 \ (\bmod \ 4)$, let $\chi : (\mathbb Z / D\mathbb Z)^\times \rightarrow \{ \pm 1 \}$ be a homomorphism given by $\chi ([m]) = (\frac{D}{m})$. Let $N \in \mathbb N$ and let $f(x, y) = ax^2 + bxy + cy^2$ be a primitive form of discriminant $D$ with $(a, N) = 1$. Then

(1) $H := \{ [m] \in (\mathbb Z / D\mathbb Z)^\times \, | \, \mbox{$m$ is $N$-represented by the principal form of discriminant $D$} \}$ is a subgroup of $\mathrm{ker}(\chi)$.

(2) $\{ [m] \in (\mathbb Z / D\mathbb Z)^\times \, | \, \mbox{$m$ is $N$-represented by $f(x, y)$} \}$ is a coset of $H$ in $\mathrm{ker}(\chi)$.
\end{lemma}

\begin{proof}
We first show that $\{ [m] \in (\mathbb Z / D\mathbb Z)^\times \, | \, \mbox{$m$ is $N$-represented by $f(x, y)$} \}$ is contained in $\mathrm{ker}(\chi)$. Let $m \in \mathbb N$ be $N$-represented by the form $f(x, y)$. Then we can write $m = d^2 m'$, where $m'$ is properly $N$-represented by $f(x, y)$ and $\chi([m]) = \chi([d])^2 \chi ([m']) = \chi ([m'])$. Thus we may assume that $m$ is properly $N$-represented by $f(x, y)$. By the first lemma on a previous page, we deduce that $b^2 - 4mc = D$ for some $b, c \in \mathbb Z$. The properties of the Kronecker symbol imply that
\[ \chi([m]) \ = \ (\frac{D}{m}) \ = \ (\frac{b^2}{m}) \ = \ 1 \]
and hence our claim is proved.

Put $d = (x, y)$. Then $m' := f(x/d, y/d)$ is properly represented by $f$ and $m' = \frac{m}{d^2}$ is relatively prime to $nN$.

(1) By applying the preceding lemma with $n = D$ and $a = 1$, we see that $H \subset \mathrm{ker}(\chi)$ is nonempty. The identities
\begin{eqnarray*}
(x^2 + ay^2)(z^2 + aw^2) & = & (xz - ayw)^2 + a(xw + yz)^2, \\
(x^2 + xy + ay^2)(z^2 + zw + aw^2) & = & (xz - ayw)^2+ (xz - ayw)(xw + yz + yw) + a(xw + yz + yw)^2
\end{eqnarray*}
show that $H$ is closed under multiplication, and hence $H$ is a subgroup of $\mathrm{ker}(\chi)$.

(2) If we apply the preceding lemma again with $n = D$, we may assume that $f(1, 0) = a$ is relatively prime to $ND$. Suppose that $D \equiv 0 \ (\bmod \ 4)$. Then $b$ is even and can be written as $2b'$, and then we have
\[ a f(x, y) \ = \ (ax + b'y)^2 - \frac{D}{4} y^2. \]
Note that $(ax + b'y, N) = 1$ if $(x, N) = 1$ and $y \equiv 0 \ (\bmod \ N)$. Since $(a, D) = 1$, it follows that the values in $(\mathbb Z / D\mathbb Z)^\times$ which are $N$-represented by $f(x, y)$ lie in the coset $[a]^{-1} H$.

Conversely, if $[c] \in [a]^{-1}H$, then $ac \equiv z^2 - \frac{D}{4}w^2 \ (\bmod \ D)$ for some $z, w \in \mathbb Z$ with $(z, N) = 1$, $w \equiv 0 \ (\bmod \ N)$. Using the above identity, it is easy to solve $f(x, y) \equiv c \ (\bmod \ D)$ with $(x, N) = 1$ and $y \equiv 0 \ (\bmod \ N)$. Thus the coset $[a]^{-1}H$ precisely consists of the values $N$-represented in $(\mathbb Z / D \mathbb Z)^\times$ by $f(x, y)$.

The case that $D \equiv 1 \ (\bmod \ 4)$ is similary by making use of the identity
\[ af(x, y) \ = \ (ax + \frac{b-1}{2}y)^2 + (ax + \frac{b - 1}{2}y)y + \frac{1 - D}{4}y^2. \]
\end{proof}

\begin{definition}
With the notations as above, we define the $N$-genus of a coset of $H$ in $\mathrm{ker}(\chi)$ to be the set consisting of all primitive forms $ax^2 + bxy + cy^2$ with $(a, N) = 1$ of discriminant $D$ which $N$-represent the values of the given coset modulo $D$.
\end{definition}

The preceding lemma immediately implies the following refinement of a previous theorem.

\begin{theorem}
Let $D \equiv 0, 1 \ (\bmod \ 4)$ be negative and let $H \subset \mathrm{ker}(\chi)$ be as in the preceding lemma. If $[a]H$ is a coset of $H$ in $\mathrm{ker}(\chi)$ and $p$ is an odd prime not dividing $D$, then $[p]$ is contained in $[a]H$ if and only if $p$ is $N$-represented by a $\Gamma_0(N)$-reduced form of discriminant $D$ in the genus of $[a]H$.
\end{theorem}

The genus containing the principal form is called the principal genus.

\begin{corollary}
(1) Let $m, N \in \mathbb N$ and $p$ an odd prime not dividing $m$. Then $p$ is $N$-represented by a form of discriminant $-4m$ in the principal genus if and only if for some integer $x$ with $(x, N) = 1$,
\[
\left\{
\begin{array}{ll}
p \equiv x^2 \mbox{ or } x^2 + m \ (\bmod \ {4m}) & \mbox{if $N$ is odd}, \\
p \equiv x^2 \ (\bmod \ {4m}) & \mbox{if $N$ is even}.
\end{array}
\right.
\]
(2) Let $m, N \in \mathbb N$ and $p$ an odd prime not dividing $1 - 4m$. Then $p$ is $N$-represented by a form of discriminant $1 - 4m$ in the principal genus if and only if for some integer $x$ with $(x, N) = 1$,
\[
p \equiv x^2 \ (\bmod \ {1 - 4m}).
\]
\end{corollary}

\begin{proof}
Note that
\[
4(x^2 + xy + my^2) \equiv (2x + y)^2 \ (\bmod \ {1 - 4m}) \ \Rightarrow \ x^2 + xy + my^2 \equiv (x')^2 \ (\bmod \ {1 - 4m}),
\]
where $x' := x + 2^{-1}y \in (\mathbb Z / (1 - 4m)\mathbb Z)^\times$. 
\end{proof}

\begin{example}
Consider the case that $D = -28$ and $N = 2$. There are exactly two $\Gamma_0(2)$-reduced forms $x^2 + 7y^2$, $7x^2 + y^2$ in the $2$-genera for the discriminant $-28$. The values in $(\mathbb Z/ 28\mathbb Z)^\times$ which are $2$-represented by $x^2 + 7y^2$ are $1, 9, 25$. The values $2$-represented by $7x^2 + y^2$ in $(\mathbb Z/ 28\mathbb Z)^\times$ are $11$, $23$, $15$. Note that $\mathrm{ker}(\chi) = \{1, 9, 11, 15, 23, 25 \}$. Therefore we deduce from the preceding theorem that for any prime $p$ with $(p, 14) = 1$,
\begin{eqnarray*}
p \equiv 1, 9, 25 \ (\bmod \ {28}) & \Leftrightarrow & \mbox{$p$ is $2$-represented by $x^2 + 7y^2$}, \\
p \equiv 11, 15, 23 \ (\bmod \ {28}) & \Leftrightarrow & \mbox{$p$ is $2$-represented by $7x^2 + y^2$}.
\end{eqnarray*}
Note that $(7x^2 + y^2)(7z^2 + w^2) = (7xz - yw)^2 + 7(xw + yz)^2$ and $(7xz - yw, 2) = 1$, $xw + yz \equiv 0 \ (\bmod \ 2)$ if $(xz, 2) = 1$, $y \equiv w \equiv 0 \ (\bmod \ 2)$. From this, we further infer that for primes $p$, $q$ with $(pq, 14) = 1$,
\[ p, q \equiv 11, 15, 23 \ (\bmod \ {28}) \ \Rightarrow \ pq = x^2 + 7y^2, \ (x, 2) = 1, \ y \equiv 0 \ (\bmod \ 2). \]
In fact, there is a version of the above identity that holds for any form of discriminant $-4m$
\[
(ax^2 + 2bxy + cy^2)(az^2 + 2bzw + cw^2) \ = \ (axz + bxw + byz + cyw)^2 + m(xw - yz)^2,
\]
where $(axz + bxw + byz + cyw, N) = 1$ and $xw - yz \equiv 0 \ (\bmod \ N)$ if $(a, N) = 1$, $(xz, N) = 1$, and $y \equiv w \equiv 0 \ (\bmod \  N)$.
\end{example}

\begin{theorem}
A positive integer $m$ is $N$-represented by a primitive form $f(x, y)$ of discriminant $D < 0$ if and only if there exists a proper $\mathcal O$-ideal $\mathfrak a \in \Phi([f(x, y)])$ such that $m = \mathrm{N}(\mathfrak a)$.
\end{theorem}

\begin{proof}
($\Rightarrow$) Let $m = f(x_0, y_0)$ for some $x_0, y_0 \in \mathbb Z$ with $(x_0, N) = 1$, $y_0 \equiv 0 \ (\bmod \ N)$. Put $d = (x_0, y_0)$. Then $a := m/d^2$ is properly $N$-represented by $f(x, y)$. Without loss of generality we may assume that $f(x, y) = ax^2 + bxy + cy^2$ for some $b, c \in \mathbb Z$ by a previous lemma. Since $\Phi([f(x, y)]) = [\mathbb Za + \mathbb Za\tau]$ and $\mathrm{N}(\mathbb Za + \mathbb Z a \tau) = a$ with $\tau = \frac{-b + \sqrt{D}}{2a}$, we see that $\mathrm{N}(\mathbb Zad + \mathbb Zad\tau) = ad^2 = m$ and $\mathbb Zad + \mathbb Zad\tau \in \Phi([f(x, y)])$.

($\Leftarrow$) Let $f(x, y) = ax^2 + bxy + cy^2$ be a primitive form of discriminant $D$ with $(a, N) = 1$ and let $\tau \in \mathbb H$ satisfy $f(\tau, 1) = 0$. From the proof of the surjectivity of $\Phi$, we see that $\mathfrak a = t(\mathbb Z + \mathbb Z \tau)$ for some $t \in \mathbb Z$ with $(t, N) = 1$. Then we have $\Phi([f(x, y)]) = [\mathbb Za + \mathbb Za\tau] = [\mathfrak a]$. We claim that $m$ is $N$-represented by $f(x, y)$. Because $t\tau \in \mathfrak a \subset \mathcal O = \mathbb Z + \mathbb Za\tau$, we put $t = as$ for some $s \in \mathbb Z$. Then we have $(s, N) = 1$ and
\[ m \ = \ \mathrm{N}(\mathfrak a) \ = \ \frac{t^2}{a} \ = \ as^2 \ = \ f(s, 0). \]
\end{proof}

\begin{corollary}
Let $\mathcal O$ be an order in an imaginary quadratic field and $N \in \mathbb N$. Given a positive integer $n$, every ideal class in $C_0(\mathcal O, N)$ contains a proper $\mathcal O$-ideal whose norm is relatively prime to $n$.
\end{corollary}

\begin{proof}
It follows directly from the preceding theorem and a previous lemma.
\end{proof}

\section{The fundamental region for $\Gamma_0(p)$}

Throughout this section, let $p \geq 5$ denote a prime number. The purpose of this section is to construct the fundamental region $\mathfrak F_{\Gamma_0(p)}$ for $\Gamma_0(p)$ which satisfies the following properties:

\begin{enumerate}
\item $|\mathrm{Re}(\tau)| \leq 1/2$ for all $\tau \in \mathfrak F_{\Gamma_0(p)}$

\item the imaginary part of $\tau \in \mathfrak F_{\Gamma_0(p)}$ is maximal in the orbit $\Gamma_0(p)\tau$

\item the real part of $\tau \in \mathfrak F_{\Gamma_0(p)}$ is minimal among the points $\tau' \in \Gamma_0(p)\tau$ with $|\mathrm{Re}(\tau')| \leq 1/2$

\item for any $\tau \in \mathbb H$, there exists a unique point $\tau' \in \mathfrak F_{\Gamma_0(p)}$ such that $\tau' \in \Gamma_0(p) \tau$.
\end{enumerate}

Given a congruence subgroup, we can explicitly construct its open fundamental region consisting of points whose imaginary part is maximal in its orbit. This can be done by applying the method given in Ferenbaugh's paper \cite[Section 3]{Fer}. For a matrix $\gamma = (\begin{smallmatrix} a & b \\ c & d \end{smallmatrix})$, we write $a_\gamma := a$, $b_\gamma := b$, and so on. As in his paper, let
\begin{eqnarray*}
\mathrm{arc}(\gamma) & = & \{ \tau \in \mathbb H \, \big| \, |\tau - a_\gamma/c_\gamma| = 1/|c_\gamma| \}, \\
\mathrm{inside}(\gamma) & = & \{ \tau \in \mathbb H \, \big| \, |\tau - a_\gamma/c_\gamma| < 1/|c_\gamma| \}, \\
\mathrm{outside}(\gamma) & = & \{ \tau \in \mathbb H \, \big| \, |\tau - a_\gamma/c_\gamma| > 1/|c_\gamma| \}
\end{eqnarray*}
for any $\gamma \in \mathrm{SL}_2(\mathbb Z)$ with $c_\gamma \neq 0$. If $\gamma = \pm(\begin{smallmatrix} 1 & b \\ 0 & 1 \end{smallmatrix})$ with $b > 0$, then we define
\begin{eqnarray*}
\mathrm{arc}(\gamma) & = & \{ \tau \in \mathbb H \, \big| \, \mathrm{Re}(\tau) = b/2 \}, \\
\mathrm{inside}(\gamma) & = & \{ \tau \in \mathbb H \, \big| \, \mathrm{Re}(\tau) > b/2 \}, \\
\mathrm{outside}(\gamma) & = & \{ \tau \in \mathbb H \, \big| \, \mathrm{Re}(\tau) < b/2 \}.
\end{eqnarray*}
On the other hand, if $\gamma = \pm(\begin{smallmatrix} 1 & b \\ 0 & 1 \end{smallmatrix})$ with $b < 0$, then we define $\mathrm{arc}(\gamma)$ in the same way and reverse the definitions of $\mathrm{inside}(\gamma)$ and $\mathrm{outside}(\gamma)$.

An open fundamental region for a congruence subgroup $\Gamma$ of $\mathrm{SL}_2(\mathbb Z)$ is defined to be an open subset $R$ of $\mathbb H$ with the properties:
\begin{enumerate}
\item there do not exist $\gamma \in \Gamma$ and $\tau, \tau' \in R$ such that $\tau \neq \tau'$ and $\tau = \gamma(\tau')$

\item for any $\tau \in \mathbb H$, there exists $\gamma \in \Gamma$ such that $\gamma(\tau) \in \bar{R}$.
\end{enumerate}

\begin{theorem}
Let $\Gamma$ be a congruence subgroup containing $(\begin{smallmatrix} 1 & 1 \\ 0 & 1 \end{smallmatrix})$. Then
\[ R_\Gamma \ := \ \bigcap_{\gamma \in \Gamma - \{ \pm I \}} \mathrm{outside}(\gamma) \]
is an open fundamental region for $\Gamma$ such that

(1) $|\mathrm{Re}(\tau)| \leq 1/2$ for all $\tau \in R_\Gamma$

(2) the imaginary part of $\tau \in \overline{R_\Gamma}$ is maximal in the orbit $\Gamma \tau$.
\end{theorem}

\begin{proof}
See \cite[Theorem 3.3]{Fer} and its proof.
\end{proof}

Thus the fundamental region $\mathfrak F_{\Gamma_0(p)}$ that we try to find lies between $R_{\Gamma_0(p)}$ and $\overline{R_{\Gamma_0(p)}}$.

\begin{lemma}
\[ R_{\Gamma_0(p)} \ = \ \{ \tau \in \mathbb H \, \big| \, |\mathrm{Re}(\tau)| < 1/2, \ 
|\tau - k/p| > 1/p \mbox{ for all } k = \pm 1, \pm 2, \ldots, \pm [p/2] \}. \]
\end{lemma}

\begin{proof}
Note that
\[
\bigcap_{\gamma \in \Gamma_0(p) - \{ \pm I \} \atop c_\gamma = 0} \mathrm{outside}(\gamma) \ = \ \{ \tau \in \mathbb H \, \big| \, |\mathrm{Re}(\tau)| < 1/2 \}
\]
and that
\[
\bigcap_{\gamma \in \Gamma_0(p) - \{ \pm I \} \atop |c_\gamma| = p} \mathrm{outside}(\gamma) \ = \ \{ \tau \in \mathbb H \, \big| \, |\tau - k/p| > 1/p \mbox{ for all } k \in \mathbb Z \mbox{ with } (k, p) = 1 \}
\]
because
\[
\gamma \in \Gamma_0(p), \ c_\gamma = \pm p \ \Longrightarrow \ \mathrm{outside}(\gamma) \ = \ \{ \tau \in \mathbb H \, \big| \, |\tau \mp a_\gamma / p| > 1/p \}.
\]
Thus we see that
\[ \bigcap_{\gamma \in \Gamma_0(p) - \{ \pm I \} \atop |c_\gamma| \leq p} \mathrm{outside}(\gamma) \ = \ \{ \tau \in \mathbb H \, \big| \, |\mathrm{Re}(\tau)| < 1/2, \ |\tau - k/p| > 1/p \mbox{ for all } k = \pm 1, \pm 2, \ldots, \pm [p/2] \}. \]
Let $\tau$ be a point contained in the set right above. If $|\mathrm{Re}(\tau)| \leq 1 / 2p$, then it is clear that $\tau \in \mathrm{outside}(\gamma)$ for any $\gamma \in \Gamma_0(p)$ with $|c_\gamma| \geq 2p$. If $|\mathrm{Re}(\tau)| > 1 / 2p$, then we see that $\mathrm{Im}(\tau) > \sqrt{3}/2p$ and so we deduce that $\tau \in \mathrm{outside}(\gamma)$ for any $\gamma \in \Gamma_0(p)$ with $|c_\gamma| \geq 2p$ because every point in $\mathrm{inside}(\gamma)$ has its imaginary part less than $1 / 2p$.
\end{proof}

%

Let $S_p = \{ \pm 1, \pm 2, \ldots, \pm \frac{p - 1}{2} \}$. For any $x \in \mathbb Z$ with $(x, p) = 1$, we denote by $x^{-1} \in S_p$ (respectively, $\langle x \rangle \in S_p$) the unique integer such that $xx^{-1} \equiv 1 \pmod p$ (respectively, $x \equiv \langle x \rangle \pmod p)$. For any $k \in S_p$, we define
\[
\gamma_k \ = \ \begin{pmatrix}
k & \frac{kk^{-1} - 1}{p} \\ p & k^{-1}
\end{pmatrix} \ \in \ \Gamma_0(p).
\]
Then we have shown in the above lemma that
\[ R_{\Gamma_0(p)} \ = \ \{ \tau \in \mathbb H \, \big| \, |\mathrm{Re}(\tau)| < 1/2 \} \bigcap \ \big(
\bigcap_{k \in S_p} \mathrm{outside}(\gamma_k) \big). \]

Notice that for any $k \in S_p$, the matrix $\gamma_k$ maps $\mathrm{arc}(\gamma_{-k^{-1}})$ onto $\mathrm{arc}(\gamma_k)$ taking endpoints to endpoints as follows:
\[ \gamma_{k}\Big(\frac{-k^{-1} \pm 1}{p}\Big) \ = \ \frac{k \mp 1}{p}. \]
Observe also that an elliptic point of order $2$ can occur only at the top points of the arcs and that an elliptic point of order $3$ can occur only at the points where the lines and the arcs intersect, namely the points $\frac{k}{p} - \frac{1}{2p} + \frac{\sqrt{3}}{2p}i$ for $k \in S_p - \{ 1 \}$.

From these, we deduce that the top point of $\mathrm{arc}(\gamma_k)$ is an elliptic point of order $2$ if and only if $\mathrm{arc}(\gamma_{-k^{-1}})$ $=$ $\mathrm{arc}(\gamma_k)$ if and only if $k^2 \equiv -1 \pmod p$. So if we set
\[ E_p^{(2)} \ := \ \{ k \in S_p \, | \, k^2 \equiv -1 \pmod p \}, \]
then the points $\frac{k}{p} + \frac{1}{p}i$ with $k \in E_p^{(2)}$ are all the inequivalent elliptic points of order $2$ for $\Gamma_0(p)$. We further observe that for such a $k$, the matrix $\gamma_k$ maps $\mathrm{arc}(\gamma_k)$ onto itself with the direction reversed. Hence, if $\tau \in \partial R_{\Gamma_0(p)} \cap \mathrm{arc}(\gamma_k)$ for some $k \in E_p^{(2)}$, then we need to choose $\tau$ such that $\mathrm{Re}(\tau) \leq \frac{k}{p}$ in order to construct $\mathfrak F_{\Gamma_0(p)}$ satisfying the prescribed property (3). On the other hand, if $\tau \in \partial R_{\Gamma_0(p)} \cap (\mathrm{arc}(\gamma_{-k^{-1}}) \cup \mathrm{arc}(\gamma_k))$ for some $k \in S_p - E_p^{(2)}$, then we need to choose $\tau$ such that $\tau \in \mathrm{arc}(\gamma_{k_{(2)}})$, where
\[ k_{(2)} \ := \ \min \{ k, \, -k^{-1} \}. \]
For example, if $k = 1$, then $k_{(2)} = -1$ and so we discard the points $\tau$ with $|\tau - 1/p| = 1/p$.

Now it remains to determine which points $\frac{k}{p} - \frac{1}{2p} + \frac{\sqrt{3}}{2p}i$ with $k \in S_p - \{ 1 \}$ are equivalent to each other under $\Gamma_0(p)$. For any $k \in S_p$ we see by a straightforward computation that
\[ \gamma_{k}\Big( \frac{1 - k^{-1}}{p} - \frac{1}{2p} + \frac{\sqrt{3}}{2p}i \Big) \ = \ \frac{k}{p} - \frac{1}{2p} + \frac{\sqrt{3}}{2p}i. \]
So the point $\frac{k}{p} - \frac{1}{2p} + \frac{\sqrt{3}}{2p}i$ is $\Gamma_0(p)$-equivalent to the point $\frac{1 - k^{-1}}{p} - \frac{1}{2p} + \frac{\sqrt{3}}{2p}i$ for any $k \in S_p$. Notice that the map $f: S_p - \{ 1 \} \longrightarrow S_p - \{ 1 \}$ taking $k$ to $\langle 1 - k^{-1} \rangle$ satisfies $f^3 = \mathrm{id}$, from which we infer that for any $k \in S_p - \{ 1 \}$ we have either $f(k) = k$ or ($k \neq f(k)$ and $f(k) \neq f^2(k)$ and $k \neq f^2(k)$). We see that
\begin{eqnarray*}
f(k) = k & \Longleftrightarrow & k + k^{-1} \equiv 1 \pmod p \\
& \Longleftrightarrow & k + k^{-1} = 1 \mbox{ or } 1 - p \\
& \Longleftrightarrow & k + k^{-1} = 1
\end{eqnarray*}
because $|k + k^{-1}| \leq p - 2$ by the assumption $p \geq 5$. Thus if $k^2 - k + 1 \equiv 0 \pmod p$, then $\gamma_k$ fixes the point $\frac{k}{p} - \frac{1}{2p} + \frac{\sqrt{3}}{2p}i$. On the other hand, if $k^2 - k + 1 \not\equiv 0 \pmod p$, then we see that the three distinct points $\frac{k}{p} - \frac{1}{2p} + \frac{\sqrt{3}}{2p}i$, $\frac{\langle 1 - k^{-1} \rangle}{p} - \frac{1}{2p} + \frac{\sqrt{3}}{2p}i$, $\frac{\langle 1 - (1 - k^{-1})^{-1} \rangle}{p} - \frac{1}{2p} + \frac{\sqrt{3}}{2p}i$ are equivalent to each other under $\Gamma_0(p)$. Therefore, if we set
\[ E_p^{(3)} \ := \ \{ k \in S_p \, | \, k^2 - k + 1 \equiv 0 \ (\bmod \ p) \}, \]
then the points $\frac{k}{p} - \frac{1}{2p} + \frac{\sqrt{3}}{2p}i$ with $k \in E_p^{(3)}$ are all the inequivalent elliptic points of order $3$ for $\Gamma_0(p)$.
Moreover, if $\tau$ is one of the three points above for some $k \in S_p - (\{ 1 \} \cup E_p^{(3)})$, then $\tau$ must be chosen as $\tau = \frac{k_{(3)}}{p} - \frac{1}{2p} + \frac{\sqrt{3}}{2p}i$, where
\[ k_{(3)} \ := \ \min \{ k, \, \langle 1 - k^{-1} \rangle, \, \langle 1 - (1 - k^{-1})^{-1} \rangle \}. \]
For example, if $k = -1$, then $k_{(3)} = -\frac{p-1}{2}$ and so the point $\tau = -\frac{1}{2} + \frac{\sqrt{3}}{2p}i$ must be chosen and the points $\tau = \pm\frac{3}{2p} + \frac{\sqrt{3}}{2p}i$ must be discarded.

Combining all of the above discussion, we conclude that $\mathfrak F_{\Gamma_0(p)}$ is the set of all points $\tau \in \mathbb H$ satisfying the following properties:
\begin{enumerate}
\item $|\mathrm{Re}(\tau)| \leq 1/2$

\item $|\tau - k/p| \geq 1/p$ for all $k \in S_p$

\item $|\mathrm{Re}(\tau)| = 1/2 \Longrightarrow \mathrm{Re}(\tau) = -1/2$

\item $|\tau - 1/p| \neq 1/p$

\item $|\tau - k/p| = 1/p$ for some $k \in E_p^{(2)}$ $\Longrightarrow$ $\mathrm{Re}(\tau) \leq k/p$

\item $|\tau - k/p| = 1/p$ for some $k \in S_p - (\{ \pm 1 \} \cup E_p^{(2)})$ $\Longrightarrow$ $\mathrm{Re}(\tau) \leq (2k_{(2)} + 1)/ 2p$

\item $\tau \neq (2k - 1) / 2p + i\sqrt{3}/2p$ for all $k \in S_p - (\{ 1 \} \cup E_p^{(3)})$ with $k \neq k_{(3)}$.
\end{enumerate}


%

\end{document}